\documentclass[11pt]{amsart}

\usepackage[ansinew]{inputenc}
\usepackage[a4paper,dvips]{geometry}
\usepackage{latexsym}
\usepackage{amsfonts}
\usepackage{amsmath,amssymb}
\usepackage{bbm}
\usepackage{color}
\usepackage{tikz}

\usepackage{graphicx}
\newtheorem{theorem}{Theorem}

\newtheorem{lemma}{Lemma}
\newtheorem{remark}{Remark}

\newcommand{\ZZ}{\mathbb{Z}}
\newcommand{\RR}{\mathbb{R}}
\newcommand{\CC}{\mathbb{C}}

\newcommand{\ba}{\mathbf{a}}

\title{{\large On the symmetry of finite sums of exponentials II}}

\author{Florian Pausinger}
\address{FP: TWT GmbH Science \& Innovation, Ernsthaldenstr. 17, 70565 Stuttgart, Germany, and School of Mathematics \& Physics, Queen's University Belfast, BT7 1NN, Belfast, United Kingdom.}
\email{ext.florian.pausinger@twt-gmbh.de, f.pausinger@qub.ac.uk}

\author{Dimitris Vartziotis}
\address{DV: TWT GmbH Science \& Innovation, Ernsthaldenstr. 17, 70565 Stuttgart, Germany, and NIKI Ltd. Digital Engineering, Research Center, Ethnikis Antistasis 205, 45500 Katsikas, Ioannina, Greece.}
\email{dimitris.vartziotis@twt-gmbh.de, dimitris.vartziotis@nikitec.gr}

\date{}

\begin{document}

\maketitle


\begin{abstract}
In this note we extend our study of the rich geometry of the graph of a curve defined as the weighted sum of two exponentials. Let $\gamma_{a,b}^{s}: [0,1] \rightarrow \mathbb{C}$ be defined as
$$\gamma_{a,b}^s(t) = (1-s) \exp(2 \pi i a t) + (1+s) \exp(2\pi i b t) $$
in which $1\leq a < b$ are two positive integers and $s \in [-1,1]$.
In the first part we determined the symmetry groups of the graphs of $\gamma_{a,b}:=\gamma_{a,b}^0$. The main aim of this note is to study the continuous transition of the graph of the curve when $s$ changes from $-1$ to $1$. 
As a main result we determine the winding numbers $\mathrm{wind}(\gamma_{a,b}^s,0)$ for $s \in [-1,1] \setminus \{ 0 \}$ as well as the set of cusp points of each such curve. This sheds further light on our initial symmetry result and provides more non-trivial albeit easy-to-state examples of advanced concepts of geometry and topology.
\end{abstract}


\section{Introduction}

Parametrised curves play a crucial role in applied mathematics as a tool to describe trajectories of objects. In the context of dynamical systems and bifurcation theory curves and, in particular, singular points of curves, are a versatile tool to describe complex behavior or change of behavior of a system. The abundance of applications of curves and manifolds motivated the development of a rich theory known as singularity theory \cite{arnold, bruce}. Seminal contributions came from H. Whitney \cite{whitney} who studied critical points of mappings as well as R. Thom \cite{zeeman} who developed catastrophe theory. J. Callahan gives a very readable introduction into these fascinating topics in his two papers on singularities and plane maps \cite{callahan, callahan2}.
The aim of our work is to give simple, but non-trivial illustrations of complicated concepts from geometry and topology. In the spirit of our first paper \cite{paper1} in which we gave examples of graphs with arbitrary symmetry groups, we focus now on generating curves with an arbitrary number of cusp points. Cusp points are an important type of singularity and often signify a transition or bifurcation in a system as we also see in our examples. We believe it is important to develop a good understanding of such singularities and, therefore, we provide simple and accessible toy models to illustrate this fundamental concept.

To introduce our model, let $\ba = (a_1, \ldots, a_m)$ denote a vector of positive integers $a_1, \ldots, a_m$ with $m\geq 2$ and let $\gamma_{\ba}: \RR \rightarrow \CC$ be the (closed) curve defined as
\begin{equation} \label{def}
\gamma_{\ba}(t) = \exp(2\pi i a_1 t) + \ldots + \exp(2\pi i a_m t) = \sum_{j=1}^m \exp(2\pi i a_j t).
\end{equation}
The function $\gamma_{\ba}$ is one-periodic, i.e., $\gamma_{\ba}(t)= \gamma_{\ba}(t+1)$ for all $t \in \RR$, since $\exp(2 \pi i a t)$, $t \in [0,1]$, $a \in \ZZ$, is a circle in the complex plane.
Moreover, the integers in $\ba$ need not be distinct. If an integer appears more than once in $\ba$ then $\gamma_{\ba}$ corresponds to a weighted sum of exponentials (with integer weights).
The graph of $\gamma_{\ba}$ can get quite involved for an arbitrary choice of parameters; see Figure \ref{fig1}. In a recent paper also infinite sums of such exponentials have been studied \cite{dv}.

\begin{figure}[h!]
\begin{center}
\includegraphics[scale=0.5]{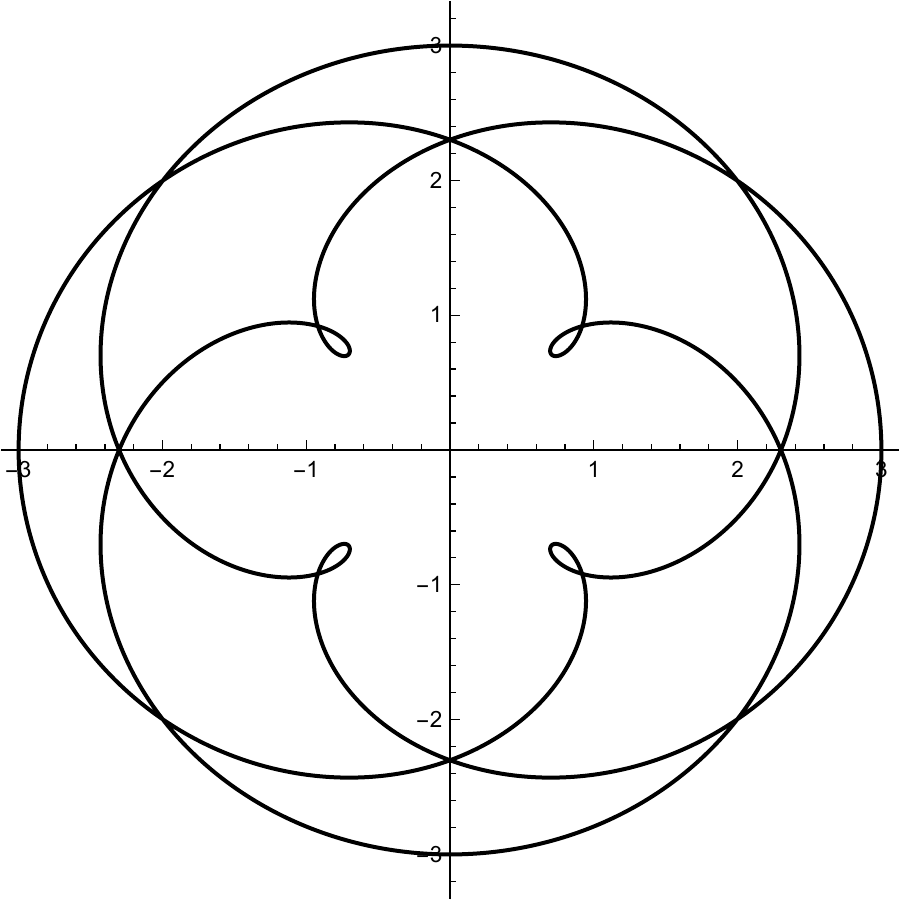}\
\includegraphics[scale=0.5]{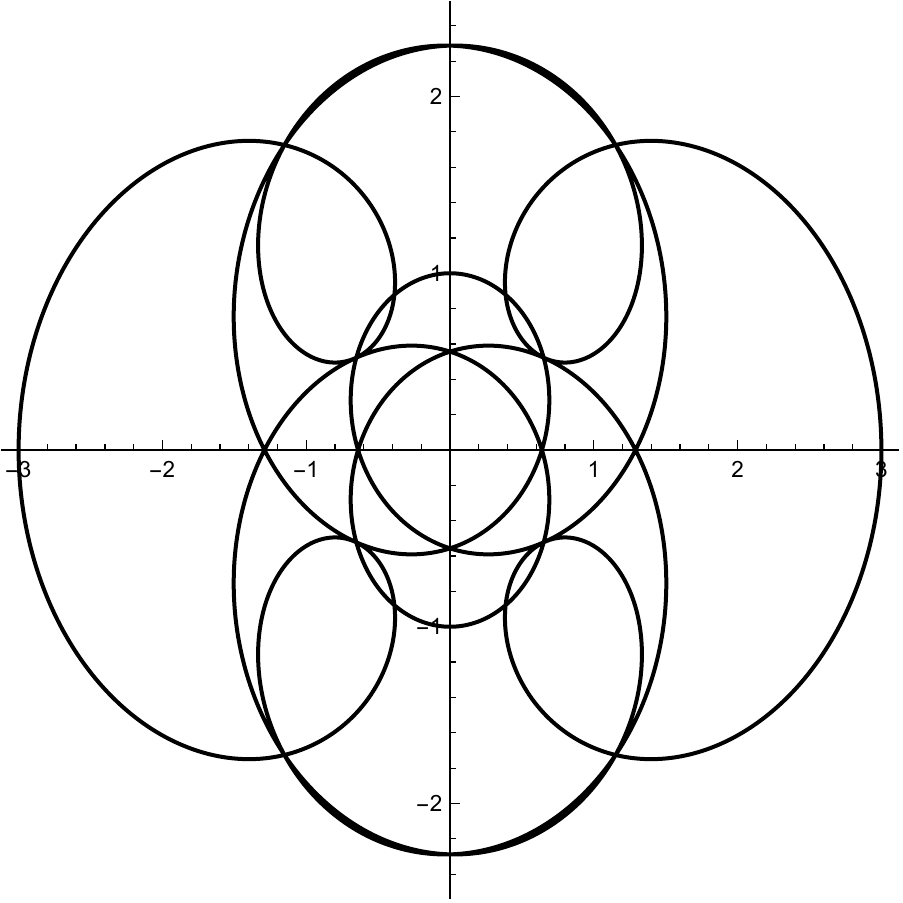} \
\includegraphics[scale=0.5]{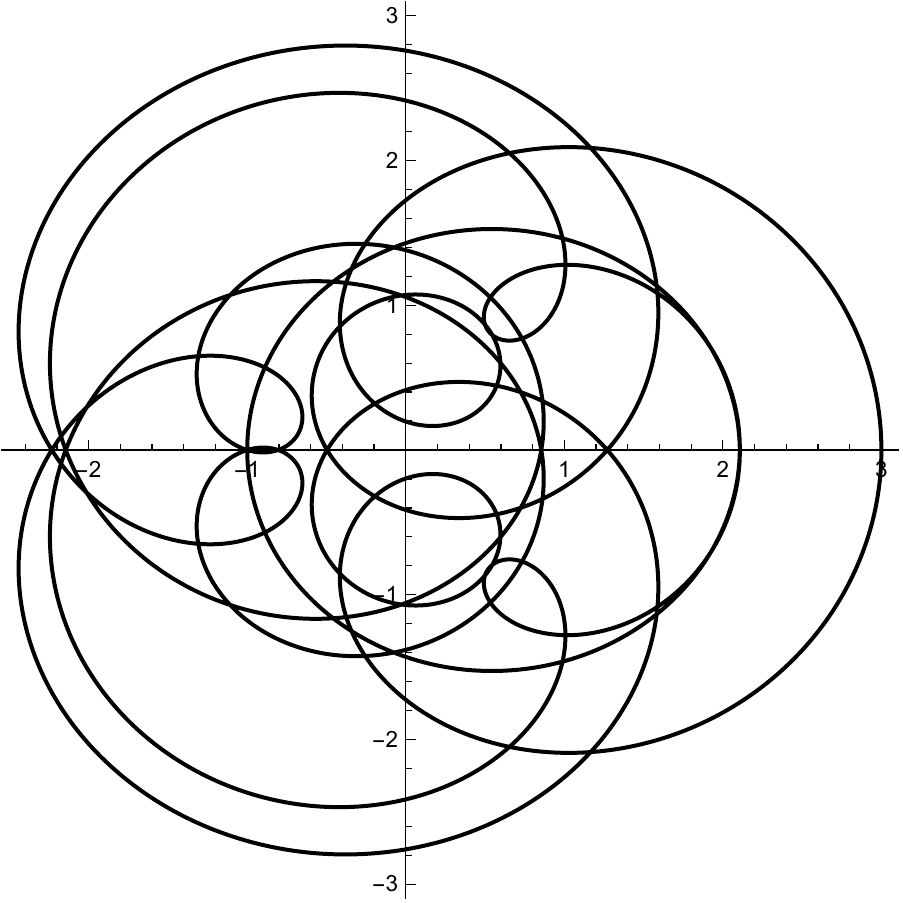} \
\end{center}
\caption{Illustration of different graphs. From left to right: $\gamma_{3,3,7}, \gamma_{5,15,45}$ and $\gamma_{2,7,13}$.} \label{fig1}
\end{figure}

In \cite{paper1} we focused on the simplest case of graphs of sums of two exponentials. We were able to determine the symmetry groups of such graphs and find all points of self intersection. Interestingly, we could show that all points of self intersection lie at equidistant parameter values $t$; i.e. $\gamma_{a,b}(t) = \gamma_{a,b}(t')$ if and only if there are integers $0 \leq j, j' \leq b^2-a^2$ with $t=j/(b^2-a^2)$ and $t'=j'/(b^2-a^2)$. 

We extend our study by introducing continuous weights. Let $s \in [-1,1]$ be a real parameter. For integers $1\leq a < b$ and a real number $t \in [0,1]$ we define a family of curves
\begin{equation} \label{def2}
\gamma_{a,b}^s (t) = (1-s) \exp(2\pi i a t) + (1+s) \exp(2\pi i b t).
\end{equation}
Note that $\gamma_{a,b} (t) = \gamma_{a,b}^0 (t)$, i.e. we define a continuous transition from $2 \exp(2\pi i a t)$ via $\gamma_{a,b}(t)$ to $2 \exp(2\pi i b t)$ such that the sum of weights is always 2. See Figure \ref{fig2} for an illustration of $\gamma_{1,3}^s(t)$.

\begin{figure}[h!]
\begin{center}
\includegraphics[scale=0.25]{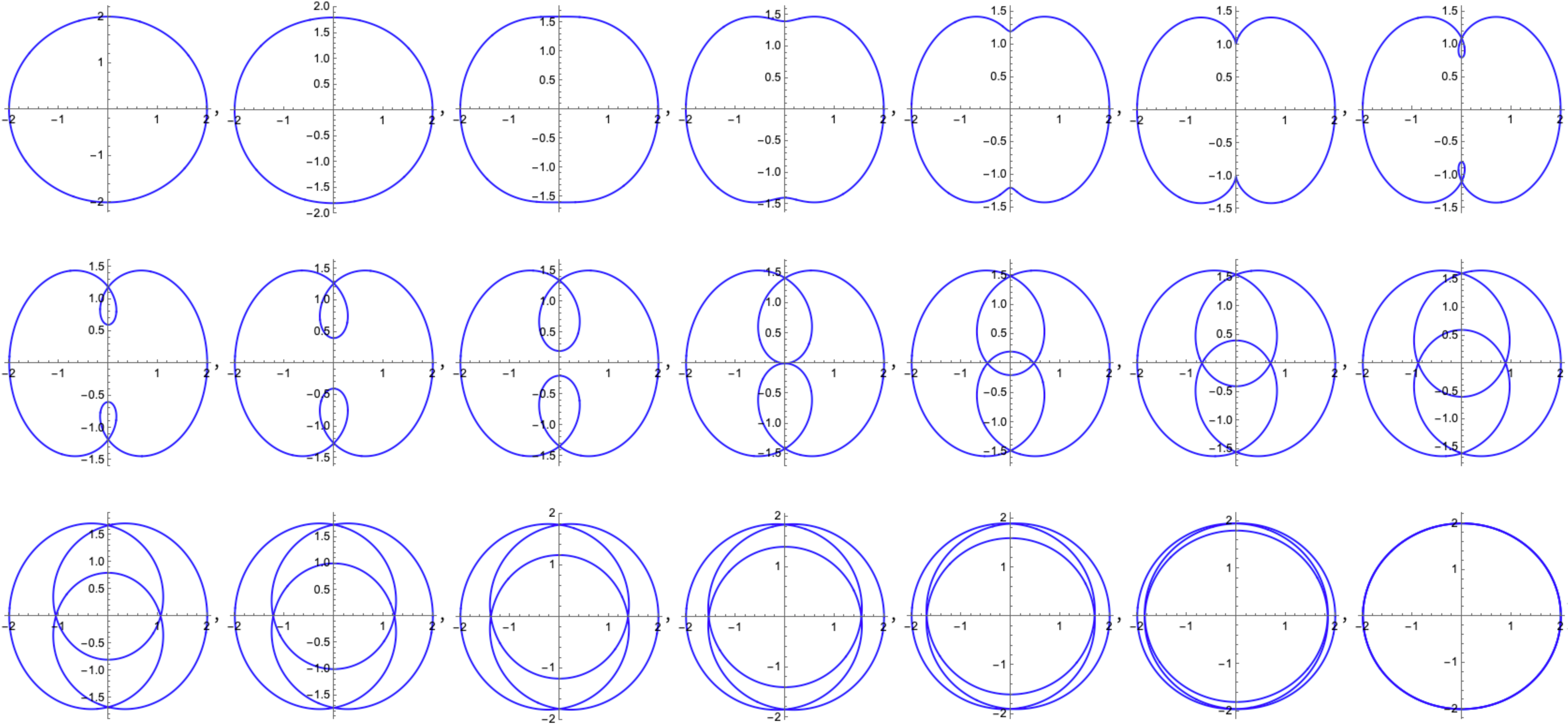}
\end{center}
\caption{Illustration of $\gamma_{1,3}^s(t)$ for $s=-1, -0.9, \ldots, 0.9, 1$.} \label{fig2}
\end{figure}
 
Our main aim is to illustrate the concept of \emph{winding numbers} as well as \emph{cusp points}. For given parameters $a$ and $b$, we first determine the winding numbers $\mathrm{wind}(\gamma_{a,b}^s, 0)$ of the corresponding curves $\gamma_{a,b}^s$. 

We recall from \cite[Lemma 2]{paper1} that $\gamma_{a,b}(t)=0$ if and only if $t=\frac{h}{2(b-a)}$ for odd integers $h \in [0, \ldots, 2(b-a)]$. 
This shows that the graph of every curve $\gamma_{a,b}$ crosses the origin $b-a$ times. 
Now, looking at Figure \ref{fig2}, we notice that $\gamma_{a,b}^s(t)$ can only be zero when $s=0$. 
This motivates us to place ourselves in the origin and investigate the winding number of every curve $\gamma_{a,b}^s(t)$ with $s \in [-1,1] \setminus \{ 0\}$; see also Figure \ref{fig3}.

\begin{theorem} \label{thm:wind}
For integers $1\leq a <b$, we have that $\mathrm{wind}(\gamma_{a,b}^s, 0) = a$ for $s \in [-1,0)$ and $\mathrm{wind}(\gamma_{a,b}^s, 0) = b$ for $s \in (0,1]$.
\end{theorem}

\begin{figure}[h!]
\begin{center}
\includegraphics[scale=0.5]{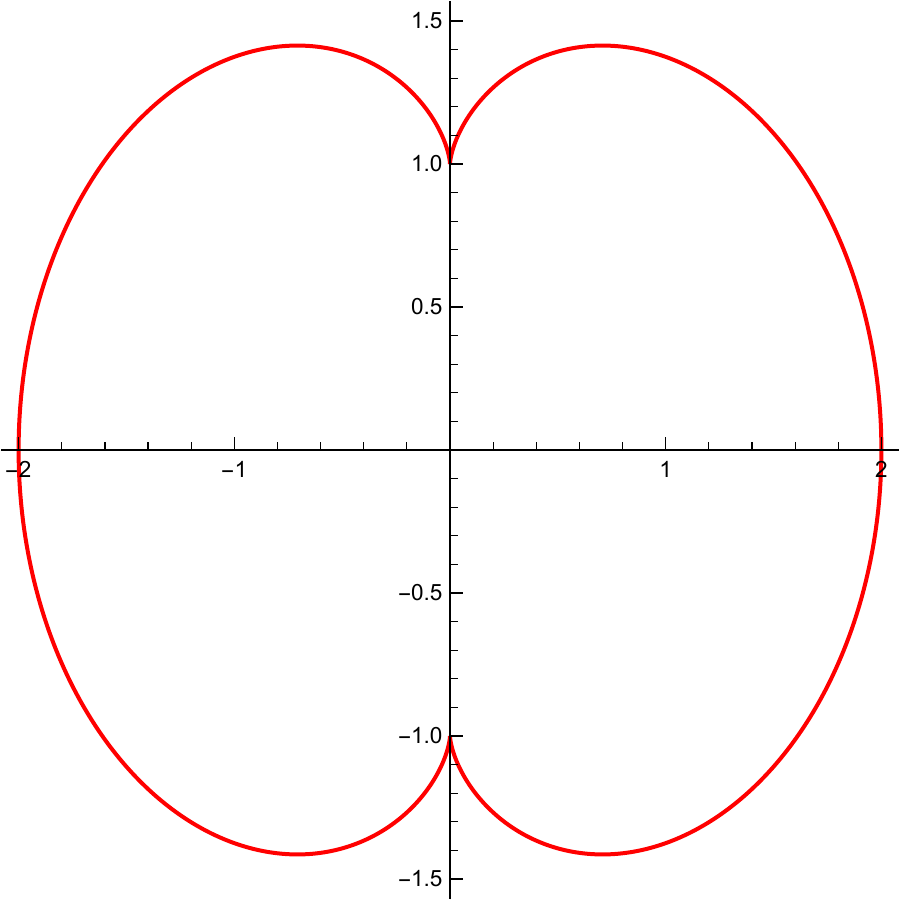}
\includegraphics[scale=0.5]{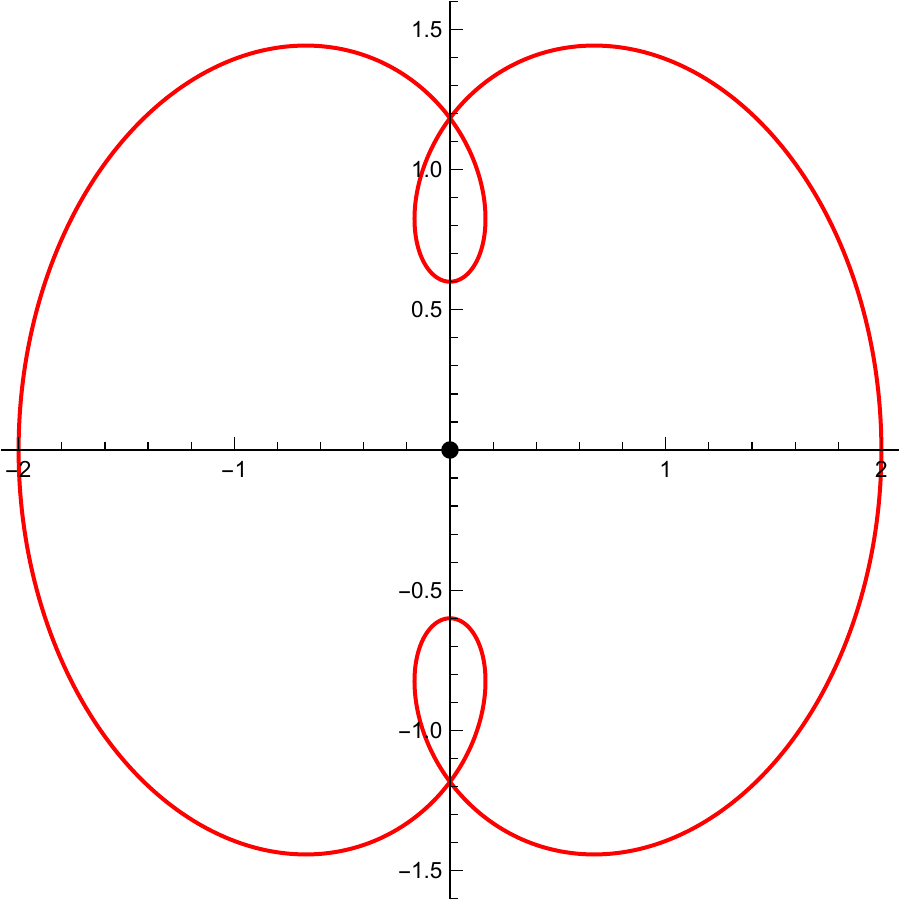}
\includegraphics[scale=0.5]{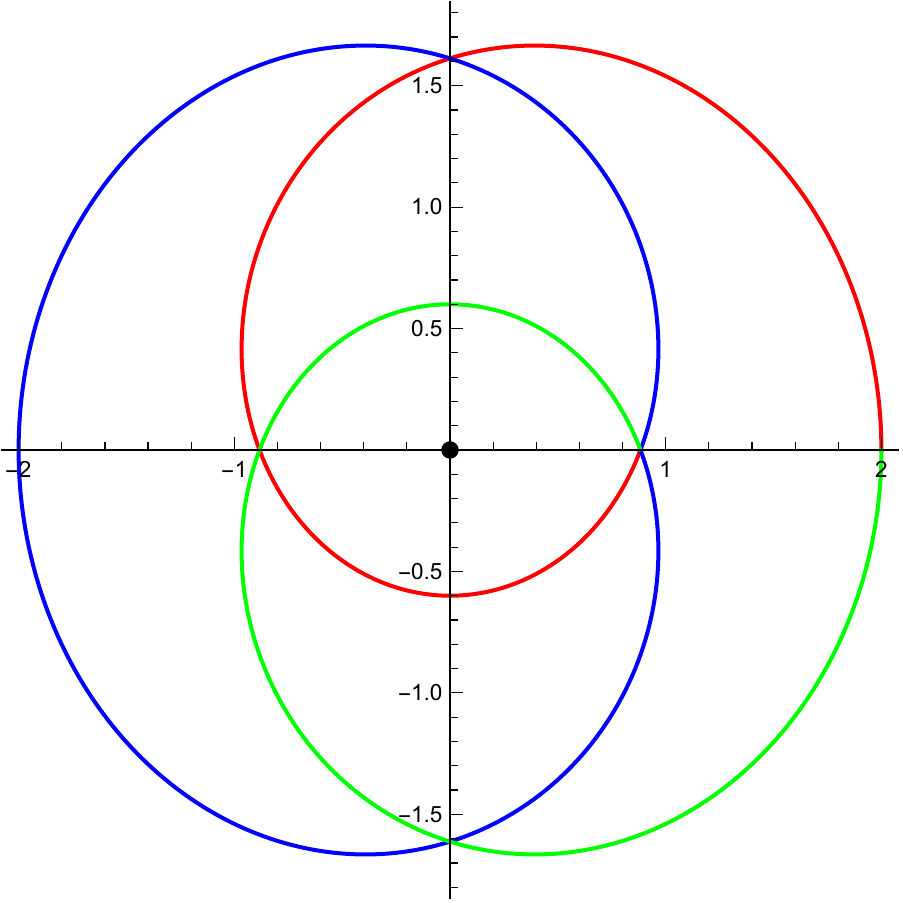}
\end{center}
\caption{Illustration of the cusp points of $\gamma_{1,3}^{-0.5}$ (left) as well as the winding numbers $\mathrm{wind}(\gamma_{1,3}^{-0.3},0)=1$ (middle) and $\mathrm{wind}(\gamma_{1,3}^{0.3},0)=3$ (right).} \label{fig3}
\end{figure}

Furthermore, looking again at Figure \ref{fig2}, we see that there exists a particular parameter $\bar{s}$ at which double points are born. Such points are called cusp points and as a second main result we determine the numbers $\bar{s}=\bar{s}(a,b)$ for given integers $a,b$ together with the parameters $t$ at which these cusp points sit.

In particular, we study the special case $a=1$ and $b=3$ in Section \ref{sec:example} and prove:
\begin{theorem}\label{thm:example}
For $a=1$, $b=3$ and $s\in [-1,1]$, the curve $\gamma_{a,b}^s$ has a cusp point if and only if $s=-1/2$ and $t=1/4$ or $t=3/4$.
\end{theorem} 
This result can be generalised. We state our observation for the general case without proof, because the calculations quickly get tedious and do not add any new ideas. We observe that for integers $1\leq a <b$ and real $s\in [-1,1]$, the curve $\gamma_{a,b}^s$ has a cusp point if and only if $s=\frac{a-b}{a+b}$ and $t=\frac{h}{2(b-a)}$ for odd integers $h \in [0, \ldots, 2(b-a)]$.

The main aim of this note is to give simple, but non-trivial examples of curves with arbitrarily many cusp points. For this reason, we provide a proof of the following general statement in Section \ref{sec:general}, which can be seen as a partial proof of the above observation and which shows how to generate curves with arbitrarily many cusp points.

\begin{theorem} \label{thm:general}
Let $a=1$ and $b$ be an integer with $a<b$. If $s=\frac{a-b}{a+b}$, then the curve $\gamma_{a,b}^s$ has a cusp point whenever $t=\frac{h}{2(b-a)}$ for odd integers $h \in [0, \ldots, 2(b-a)]$.
\end{theorem}

Our results shed further light on our symmetry result in \cite{paper1}. In particular, we see that the winding number of a single exponential of the form $\exp(2\pi k t)$ for an integer $1 \leq k$ is $k$. Now it turns out that when we continuously transition from $2\exp(2\pi a t)$ to $2\exp(2\pi b t)$ there is always a parameter $\bar{s}$ at which exactly $b-a$ many loops are born. Observing from the origin, these loops keep growing and enable the jump in the winding number of the curve when transitioning from a value $s<0$ to a value $s>0$.

\section{Preliminaries}

\subsection{Symmetry}
We observe that the arguments used to prove \cite[Corollary 1]{paper1} carry over verbatim. On the one hand, the weights $1-s$ and $1+s$ have no influence on the congruences used to establish the rotational symmetry of the graphs. On the other hand, the weights have also no influence on the argument concerning the $b-a$ points of maximal distance to the origin. Hence, we can conclude:

\begin{lemma}
For coprime integers $1 \leq a < b$ and $s \in (-1,1)$ the symmetry group of the graph of $\gamma_{a,b}^s$ is $D_{b-a}$, i.e. the dihedral group of order $2(b-a)$.
\end{lemma}

\subsection{Winding number}
The winding number $\mathrm{wind}(\gamma, z_0)$ of a \emph{closed} curve $\gamma$ in the plane around a given point $z_0$ is an integer representing the total number of times a curve travels counter-clockwise around the point. The sign of the winding number indicates the orientation of the curve, i.e. a negative winding number means traveling clockwise around a point.
Illustrating the importance of the concept, there exist different (but equivalent!) definitions in different parts of mathematics. We use the definition from complex analysis. 
If $\gamma$ is a closed curve in the complex plane parametrized by $t \in [0,1]$, the winding number of $\gamma$ around a point $z_0$ is defined for complex $z_0 \notin \gamma([0,1])$ as
$$ \mathrm{wind}(\gamma, z_0) = \frac{1}{2 \pi i} \int_0^1 \frac{\gamma'(t)}{\gamma(t) - z_0} dt. $$

It can be shown \cite[Theorem 10.10]{rudin} that the winding number is integer-valued, constant over each maximal connected subset of $\Omega=\mathbb{C} \setminus \gamma([0,1])$ and zero if $z_0$ is in the unbounded component of $\Omega$.

\subsection{Singular points of a curve in the plane}
A parametrised curve $\gamma$ in $\mathbb{R}^2$ is a function $\gamma: \mathbb{R} \rightarrow \mathbb{R}^2$ with $\gamma(t)=(x(t), y(t))$. 
The curve is closed if $\gamma(0)=\gamma(1)$.
The vector 
$$\gamma'(t) =  \left( \frac{d}{dt} x(t), \frac{d}{dt} y(t) \right)$$
is the tangent vector of the curve $\gamma$ in $t$.
Points $t \in [0,1]$ with $\gamma'(t)=0$ are called singular, those with $\gamma'(t)\neq0$ are called regular.
The parametric derivative of a curve is defined as 
$$\frac{y'(t)}{x'(t)} \quad \text{ subject to } x'(t) \neq 0.$$
Note that the parametric derivative is undefined at a singular point $t$. In particular, the parametric derivative is undefined whenever $x'(t)=0$, i.e., whenever we have a vertical tangent. If $y'(t)=0$ and $x'(t) \neq 0$, i.e., if we have a horizontal tangent, then the parametric derivative is zero.

\section{Winding numbers}

In order to prove Theorem \ref{thm:wind} we need the following lemma which is an application of the Residue Theorem.

\begin{lemma} \label{lem:integral} For real numbers $\alpha, \beta$, with $\beta >0$ we have that
$$ \int_0^1 \frac{1}{\beta + \alpha \cdot \exp(2 \pi i t)} dt = \begin{cases} \frac{1}{\beta} & \text{if} \ \ \beta > |\alpha | \\ 0 & \text{if} \ \ \beta < |\alpha| \end{cases}.$$
\end{lemma}

\begin{proof}
Let us integrate the complex function
$$ f(z) = \frac{1}{\alpha+z}$$
over the the curve $|z| = \beta$. For an integer $b \neq 0$ and real $\beta >0$, we can parametrize the curve using
$$ \gamma(t) = \beta e^{2 \pi i b t} \qquad \mbox{for}~0 \leq t \leq 1.$$
Note that this curve wraps $b$ times around the origin (the sign of $b$ indicates the orientation, i.e. clockwise or counter-clockwise).

Then
$$ \int_{|z| = \beta} f(z) = \int_0^1 f(\gamma(t)) \gamma'(t) dt = \int_0^1 \frac{1}{\alpha + \beta e^{2 \pi i b t}} \beta 2 \pi i b e^{2\pi i b t} dt.$$
This integral can also be written as
$$  \int_0^1 \frac{1}{\alpha + \beta e^{2 \pi i b t}} \beta 2 \pi i b e^{2\pi i b t} dt =  2 \pi i b \beta \int_0^1 \frac{1}{\alpha e^{-2b\pi i t}+ \beta }   dt$$
We can now evaluate 
$$ \int_{|z| = \beta} f(z)  \qquad \mbox{using the residue theorem.}$$
The function $f$ has a single residue in $z = -\alpha$.  This means that if $\beta > |\alpha|$, then the circular integral captures the residue. 
We have that $\mathrm{Res}(f,-\alpha)=1$ and $\mathrm{wind}(\gamma, -\alpha)=b$. Hence, by the residue theorem
$$ \int_{|z| = \beta} f(z) = 2 \pi i b \Rightarrow \frac{1}{\beta} = \int_0^1 \frac{1}{\alpha e^{-2b\pi i t}+ \beta } dt.$$

If $\beta < |\alpha|$, then the curve does not enclose anything and the integral evaluates to 0.
\end{proof}

\begin{proof}[Proof of Theorem \ref{thm:wind}]
The winding number about $z_0 \notin \gamma([0,1])$ of a curve $\gamma=\gamma_{a,b}^s$ parametrized by $t \in [0,1]$ is defined as
$$\mathrm{wind}(\gamma, z_0) = \frac{1}{2 \pi i} \int_{0}^1 \frac{\gamma'(t)}{\gamma(t) - z_0} dt.$$
Setting $z_0 =0$ and using the definition of $\gamma_{a,b}^s$ we get for $s \in [-1,1] \setminus \{ 0\}$:
\begin{align*}
\mathrm{wind}(\gamma,0) &= \frac{1}{2 \pi i} \int_{0}^1 \frac{(1-s) 2\pi i a \exp(2\pi i a t) + (1+s) 2\pi i b \exp(2\pi i b t)}{(1-s) \exp(2\pi i a t) + (1+s) \exp(2\pi i b t)} dt \\
&= \int_{0}^1 \frac{(1-s) a \exp(2\pi i a t)}{(1-s) \exp(2\pi i a t) + (1+s) \exp(2\pi i b t)} dt +\int_{0}^1 \frac{ (1+s) b \exp(2\pi i b t)}{(1-s) \exp(2\pi i a t) + (1+s) \exp(2\pi i b t)} dt \\
&=  a (1-s) \int_{0}^1 \frac{1}{(1-s) + (1+s) \exp(2\pi i (b-a) t)} dt \\ 
& \quad \quad \quad + b (1+s) \int_{0}^1 \frac{ 1}{(1-s) \exp(2\pi i (a-b) t) + (1+s)} dt .
\end{align*}
Applying Lemma \ref{lem:integral} gives the desired result: 
If $(1-s)>(1+s)$, i.e., if $s \in [-1,0)$, we get that
$$\mathrm{wind}_{\gamma}(0) = a (1-s) \frac{1}{1-s}=a.$$
Otherwise, $(1+s)>(1-s)$i.e., if $s \in (0,1]$, we get that 
$$\mathrm{wind}_{\gamma}(0) = b (1+s) \frac{1}{1+s}=b.$$
\end{proof}

\section{Cusp points - the case $a=1$, $b=3$}
\label{sec:example}

In this section we look at the special case $a=1$ and $b=3$ (see Figure \ref{fig3}) and prove Theorem \ref{thm:example}. Working in Euclidean space, we have that $\gamma_{1,3}^s(t) = (x_s(t),y_s(t))$ with
\begin{align*}
x_s(t)&=(1-s) \cos(2\pi t) + (1+s) \cos(6 \pi t)\\
y_s(t)&=(1-s) \sin(2\pi t) + (1+s) \sin(6 \pi t)
\end{align*}
such that
\begin{align*}
x'_s(t)&=-2\pi (1-s) \sin(2\pi t) - 6\pi (1+s) \sin(6 \pi t)\\
y'_s(t)&= 2\pi (1-s) \cos(2\pi t) + 6\pi (1+s) \cos(6 \pi t).
\end{align*}
Singular points of parametrised curves can be studied via the parametric derivative. 
The parametric derivative is undefined at singular points.
In our case we get
\begin{align*}
\frac{y'_s(t)}{x'_s(t)} &= \frac{2\pi (1-s) \cos(2\pi t) + 6\pi (1+s) \cos(6 \pi t)}{-2\pi (1-s) \sin(2\pi t) - 6\pi (1+s) \sin(6 \pi t)}.
\end{align*}
We start our analysis of the parametric derivative by considering three special cases. First set $s=-1$. In this case 
\begin{align*}
\frac{4\pi \cos(2\pi t)}{-4\pi \sin(2\pi t)} & = -\cot(2\pi t).
\end{align*}
Hence, the derivative is undefined if and only if $t=0, 1/2, 1$.
However, $y'_s(t) \neq 0$ at all three values and, hence, we conclude that the curve does not have a singular point at the three values but simply a vertical tangent.
Similarly we get for $s=1$
\begin{align*}
\frac{ 12\pi \cos(6 \pi t)}{ - 12\pi \sin(6 \pi t)} &= -\cot(6\pi t).
\end{align*}
In this case, the derivative is undefined if and only of $t=0, 1/6, \ldots, 5/6, 1$.
Again, $y'_s(t) \neq 0$ at all seven values and, hence, we conclude that the curve does not have a singular point but simply at vertical tangent.

As a third special case we look at $s=-0.5$. In this case we have that
\begin{align*}
\frac{3\pi \cos(2\pi t) + 3\pi \cos(6 \pi t)}{-3\pi \sin(2\pi t) - 3\pi \sin(6 \pi t)} &=\frac{ \cos(2\pi t) + \cos(6 \pi t)}{- \sin(2\pi t) - \sin(6 \pi t)}\\
&  = - \frac{2\cos\left( \frac{2\pi t +6\pi t}{2} \right)\cos \left( \frac{2\pi t -6\pi t}{2} \right)}{2\sin\left( \frac{2\pi t +6\pi t}{2} \right)\cos \left( \frac{2\pi t -6\pi t}{2} \right)} = -\cot(4\pi t).
\end{align*}
Here the derivative is undefined if and only if $t=0, 1/4, 1/2, 3/4, 1$; see also Figure \ref{paramDerivative}.
Moreover, we have that $y'_s(1/4)=y'(3/4)=0$. Hence, there are two singular points at $t=1/4$ and $t=3/4$; see Figure \ref{fig3} (left).

Next we consider the parametric derivative for $s \in (-1,-0.5)$. Note that we can rewrite the parametric derivative as
$$\frac{1+2s-3(s+1) \cos(4\pi t)}{2+s+3(s+1)\cos(4\pi t)} \cot(2 \pi t). $$
Because of the factor $\cot(2 \pi t)$, the derivative is undefined for $t=0,1/2,1$.
In addition the derivative is undefined whenever 
$$2+s+3(s+1)\cos(4\pi t)=0.$$
We can rewrite this equation and get
$$ 4 \pi t = \arccos\left (\frac{-2-s}{3(1+s)}\right).$$
Note that the arccosine is the inverse function of the cosine in the interval $[0,\pi]$. Hence, we have to restrict to $t \in [0, 1/4]$. Each solution $\bar{t}$ gives now four solutions in $[0,1]$ due to the symmetry of the cosine; i.e. $\bar{t}, 1/2-\bar{t}, 1/2+\bar{t}, 1-\bar{t}$.

\begin{figure}[h!]
\begin{center}
\includegraphics[scale=0.5]{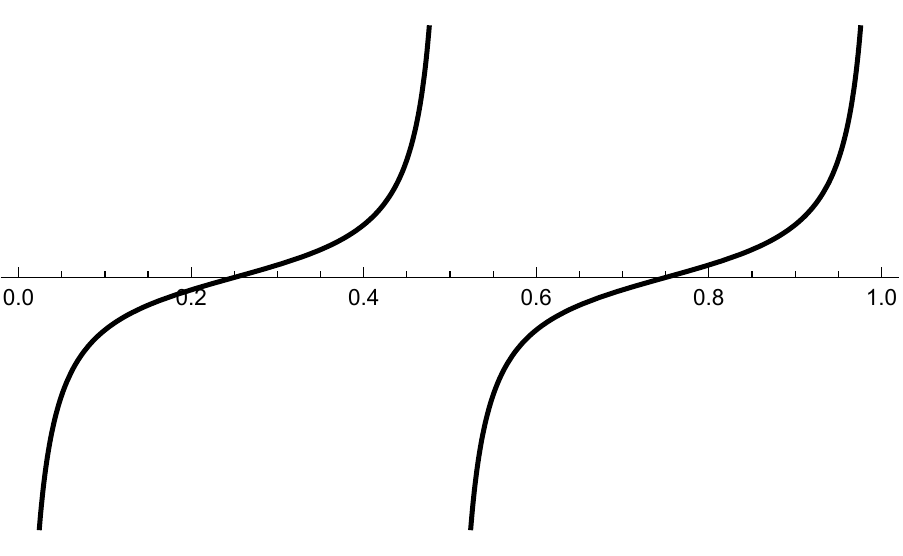} \quad
\includegraphics[scale=0.5]{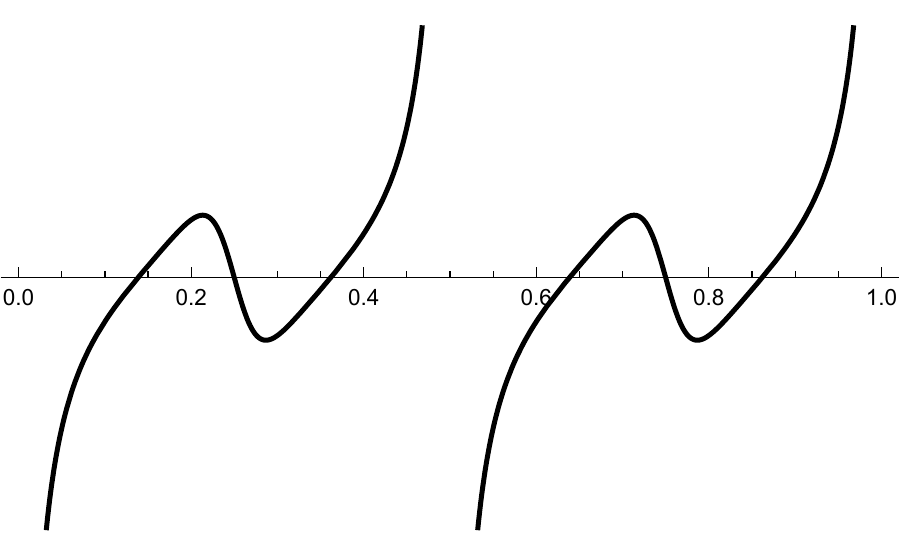} \quad
\includegraphics[scale=0.5]{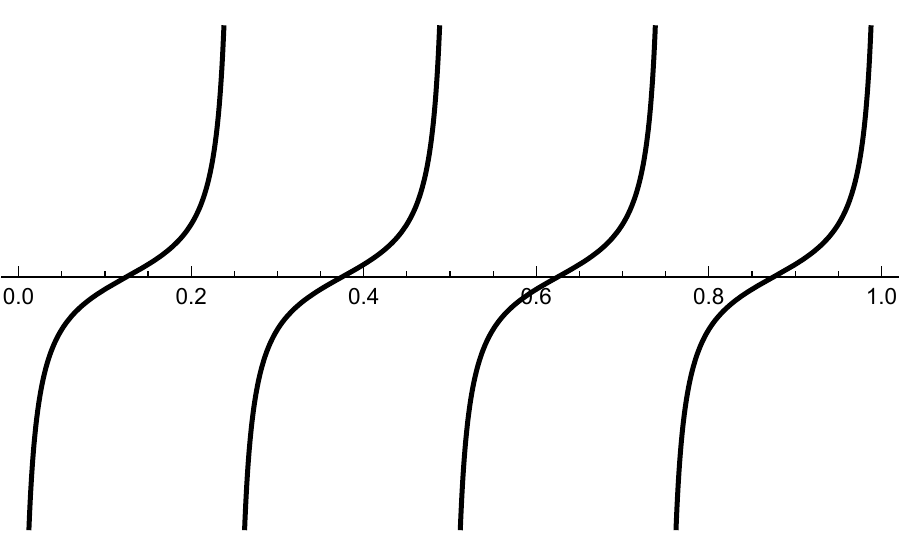} \\
\includegraphics[scale=0.5]{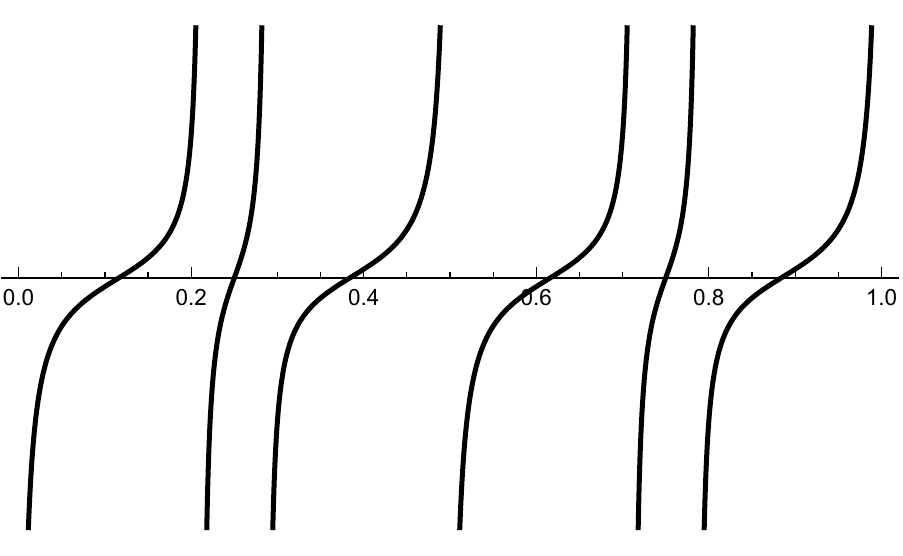}\quad
\includegraphics[scale=0.5]{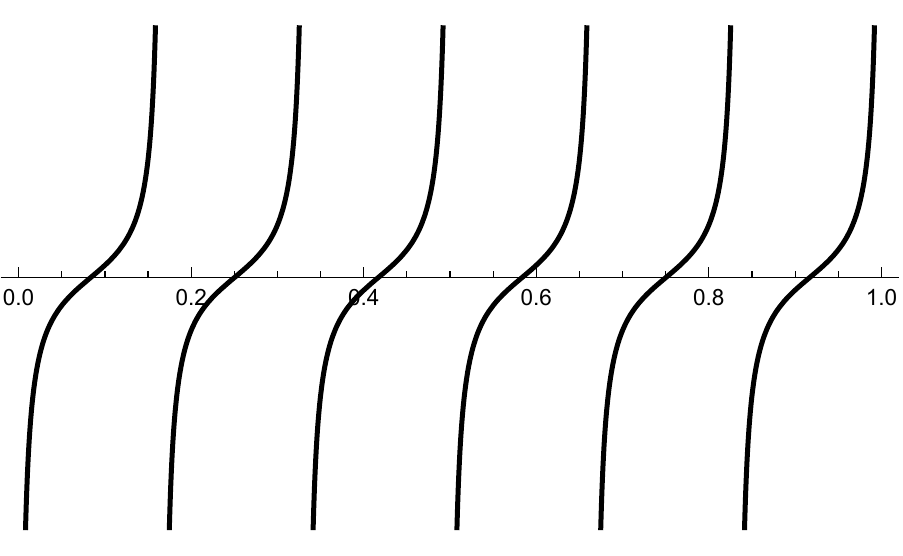}
\end{center}
\caption{Illustration of the parametric derivative of the curves $\gamma_{a,b}^s$ with $s=-1,-0.6, -0.5, 0.6, 1$} \label{paramDerivative}
\end{figure}

Moreover, the argument of the arccosine is by definition in $[-1,1]$. Therefore, we have that
$$-1 \leq \frac{-2-s}{3(1+s)} \leq 1$$
from this
$$-1 \leq \frac{-2-s}{3(1+s)} \Rightarrow - \frac{1}{2} \leq s.$$
Hence, there is no additional critical value for $s \in (-1,-0.5)$; see also Figure \ref{plot:argument}.

\begin{figure}[h!]
\begin{center}
\includegraphics[scale=0.7]{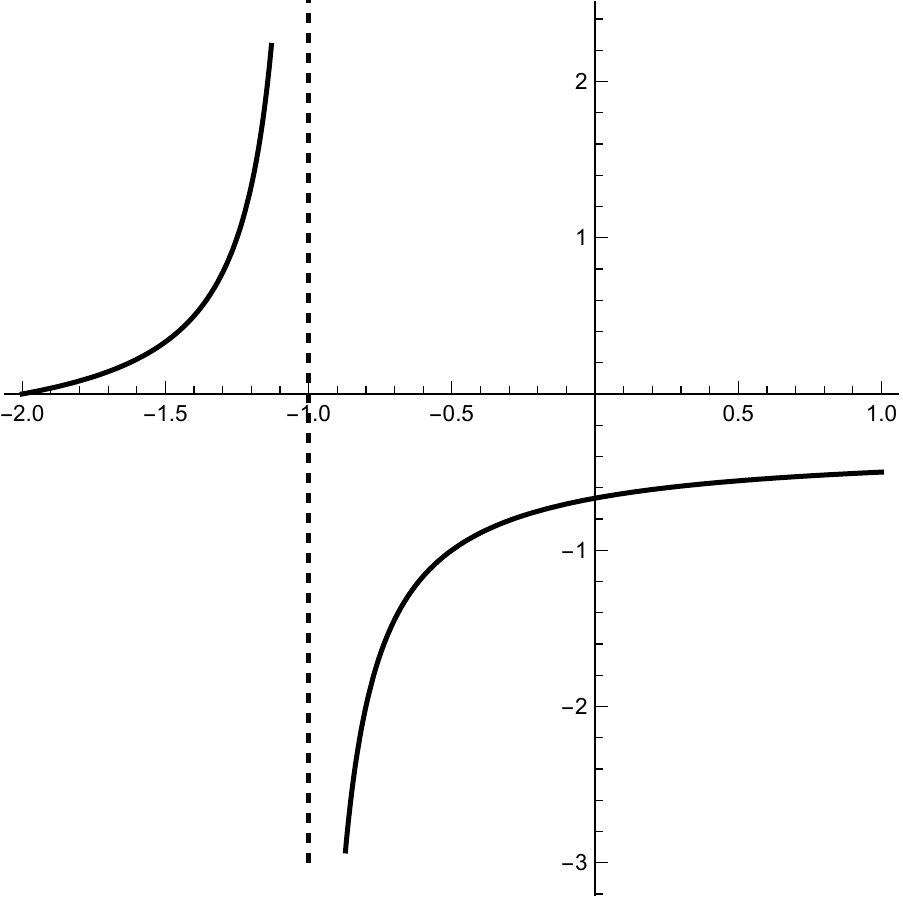}
\end{center}
\caption{Illustration of the function $\frac{-2-s}{3(s+1)}$.} \label{plot:argument}
\end{figure}

Finally, we turn to $s \in (-0.5, 1)$. In this case we obtain one solution, i.e.
$$ 2+s+3(s+1)\cos(4\pi t)=0 \Leftrightarrow t= \frac{\arccos\left (\frac{-2-s}{3(1+s)}\right)}{4 \pi}$$
and this implies that we have four additional points at which the derivative is not defined for every $s \in (-0.5,1)$. We can check for each of these additional points that $y'(t)\neq 0$. Hence, we found additional points with a vertical tangent. We summarise our findings in a plot in Figure \ref{plot:overview}.

\begin{figure}[h!]
\begin{center}
\includegraphics[scale=0.75]{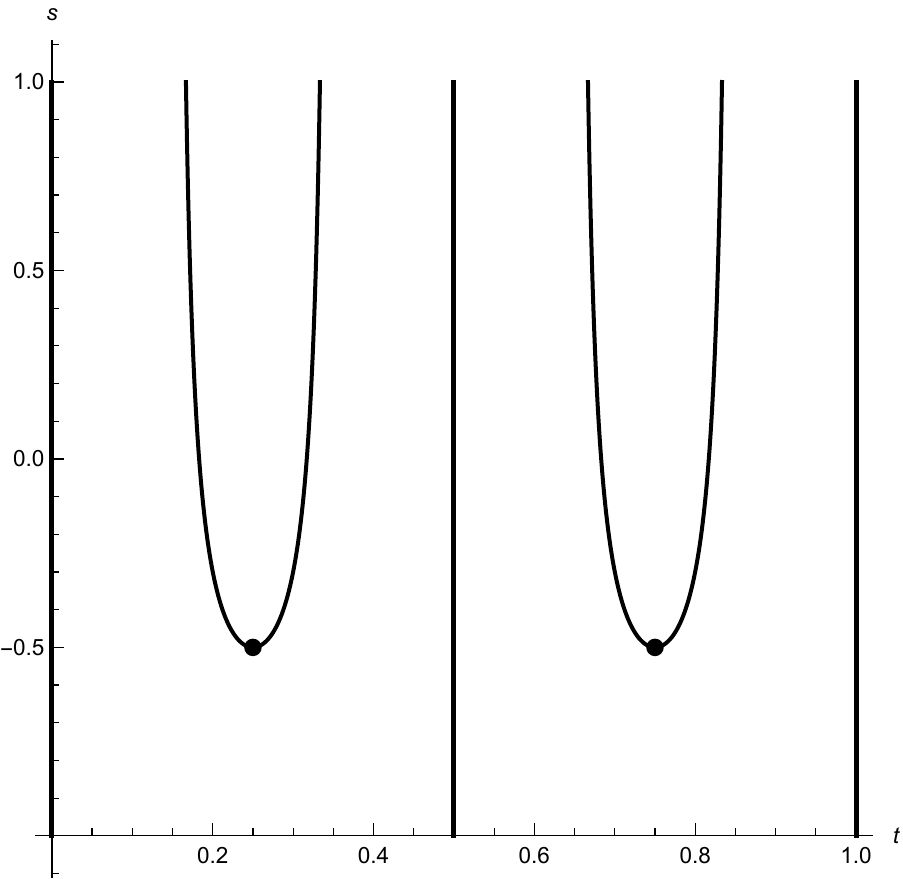}
\end{center}
\caption{Illustration of the points (black lines) of $\gamma_{1,3}^s$ at which the parametric derivative is not defined including the two cusp points (bold black points).} \label{plot:overview}
\end{figure}

We found two singular points, i.e., $(s,t)=(-0.5, 0.25)$ and $(s,t)=(-0.5,0.75)$. To confirm that these points are indeed cusp points of the curve $\gamma^{-0.5}_{1,3}$ we look at the unit tangent vectors at these points when we approach $t$ from above and from below. In case of a cusp point these vectors should have opposite sign.
Indeed, we have that
\begin{align*}
\lim_{t\rightarrow (0.25)^+} &\frac{(\gamma_{1,3}^{-0.5})'(t)}{\| (\gamma_{1,3}^{-0.5})'(t)\|} = (0,1),\\
\lim_{t\rightarrow (0.25)^-} &\frac{(\gamma_{1,3}^{-0.5})'(t)}{\| (\gamma_{1,3}^{-0.5})'(t)\|} = (0,-1).\\
\end{align*}
By symmetry, we get a similar result for $(s,t)=(-0.5,0.75)$.
Hence, we conclude that these singular points are indeed cusp points.

We separately analyse the graphs of the $x$- and $y$-coordinates of the curve $\gamma_{a,b}^s$ in a small neighborhood of $t=0.25$; see Figure \ref{coords}. We observe that the graph of the $x$-coordinate has three zeros and two local extrema for $-0.5<s$ while it only has one zero for $s \leq -0.5$. This change of sign in the $x$-coordinate finally shows that a loop is born when $s$ passes the threshold $-0.5$ and confirms the existence of a double point, i.e., a point of self intersection, whenever $-0.5<s$.

\begin{figure}[h!]
\begin{center}
\includegraphics[scale=0.5]{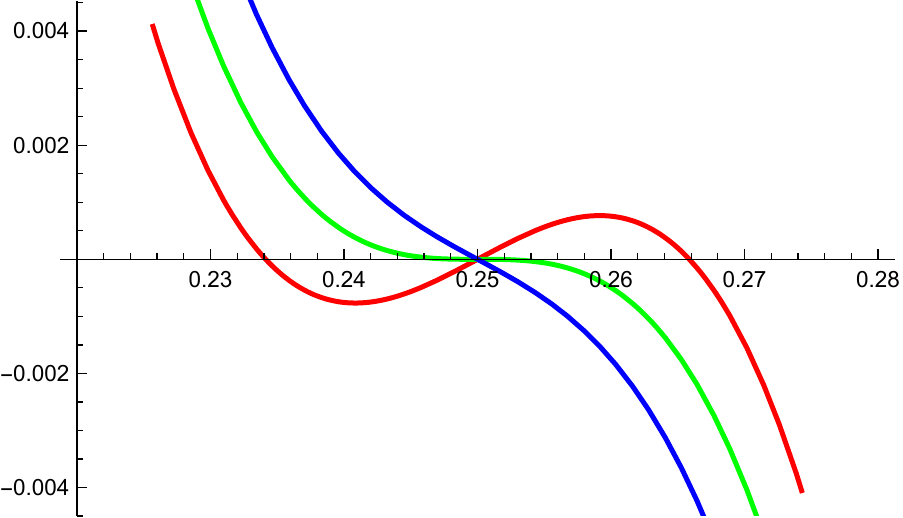} \qquad
\includegraphics[scale=0.5]{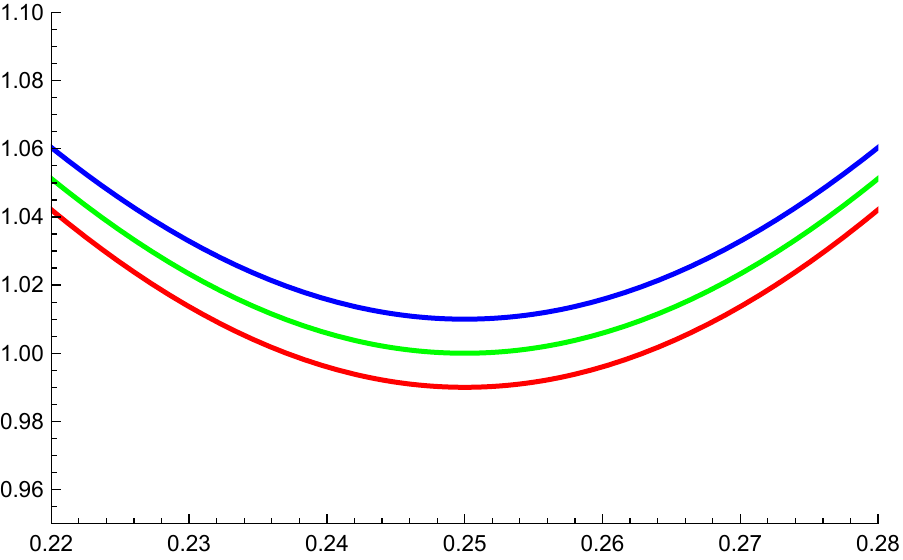}
\end{center}
\caption{Illustration of $x$-component (left) and $y$-component (right) of $\gamma_{1,3}^s(t)$ for $t \in [0.22, 0.28]$ and $s=-0.495$ (red), $s=-0.5$ (green) and $s=-0.505$ (blue).} \label{coords}
\end{figure}

This proves Theorem \ref{thm:example}.

\section{Cusp points - the general case}
\label{sec:general}

In this section we study the general case and prove Theorem \ref{thm:general}.
Geometrically, we can argue in a similar fashion in the general case as in the previous example. However, the argument becomes more delicate because the cusp points as well as the double points generally do not lie on the $y$-axis. Hence, it is harder to (formally) detect them since we do not have a vertical or horizontal tangent at those points.
In the following we extend the analysis of the previous section by adding an additional, first step. Since we know the symmetry group, i.e. $D_{b-a}$, of the graph of every curve $\gamma_{a,b}^s$ for integers $0<a<b$ and $s\in [-1,1]$, it is sufficient to only study the curves in the parameter interval $t \in [0, 1/(b-a)]$. If we find a cusp point in this interval, then every interval $[j/(b-a), (j+1)/(b-a)]$ contains a cusp point by symmetry.
Moreover, by inspection, we conjecture that the cusp points of $\gamma_{a,b}^s$ are located at
$$(s,t)=\left( \frac{a-b}{a+b}, \frac{j}{2(b-a)} \right).$$
Therefore, in order to detect them with our methods, we first rotate the graph of $\gamma_{a,b}^s$ such that the suspected cusp points lie on the $y$-axis. In a second step, we analyse the graphs exactly as in the previous section.

\begin{remark}
Note that Whitney developed a much more powerful theory in which we could view our curves $\gamma_{a,b}^s$ as a map from a manifold (with or without boundary), in our case a cylinder identified with $[-1,1] \times [0,1]$, to $\mathbb{R}^2$; see \cite{whitney} and also \cite{callaghan} for a very nice introduction. However, this goes beyond the scope of this note.
\end{remark}

In our first example we had $b-a=3-1=2$ and, hence, the two cusp points were lying exactly on the $y$-axis. If $b-a\geq 3$ we rotate the graph by the angle
$$\phi = \frac{\pi}{2} - \frac{2\pi}{2(b-a)}= \pi \left(\frac{1}{2} - \frac{1}{b-a} \right), $$
i.e., by $\pi/2$ minus the angle between the horizontal axis and the first cusp point at $t=1/(2(b-a))$; see Figure \ref{rotation}.
This rotation can be achieved by using the complex representation of our curves and multiplying with $\exp(i \phi)$ in which $\phi$ is the angle of rotation, i.e., we get
$$\exp\left(\frac{\pi i}{2} - \frac{\pi i}{b-a}\right) \cdot \gamma_{a,b}^s(t).$$
Now we restrict to $t \in [0, 1/(b-a)]$; see Figure \ref{rotation}.

\begin{figure}[h!]
\begin{center}
\includegraphics[scale=0.7]{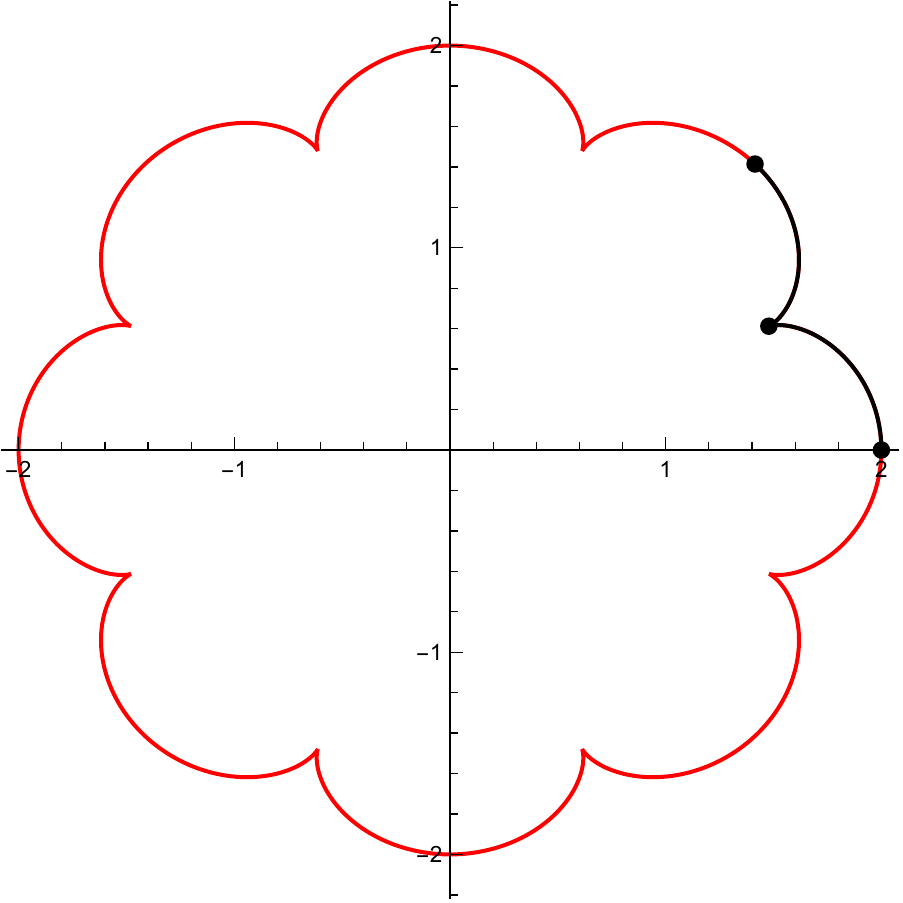} \quad \quad
\includegraphics[scale=0.7]{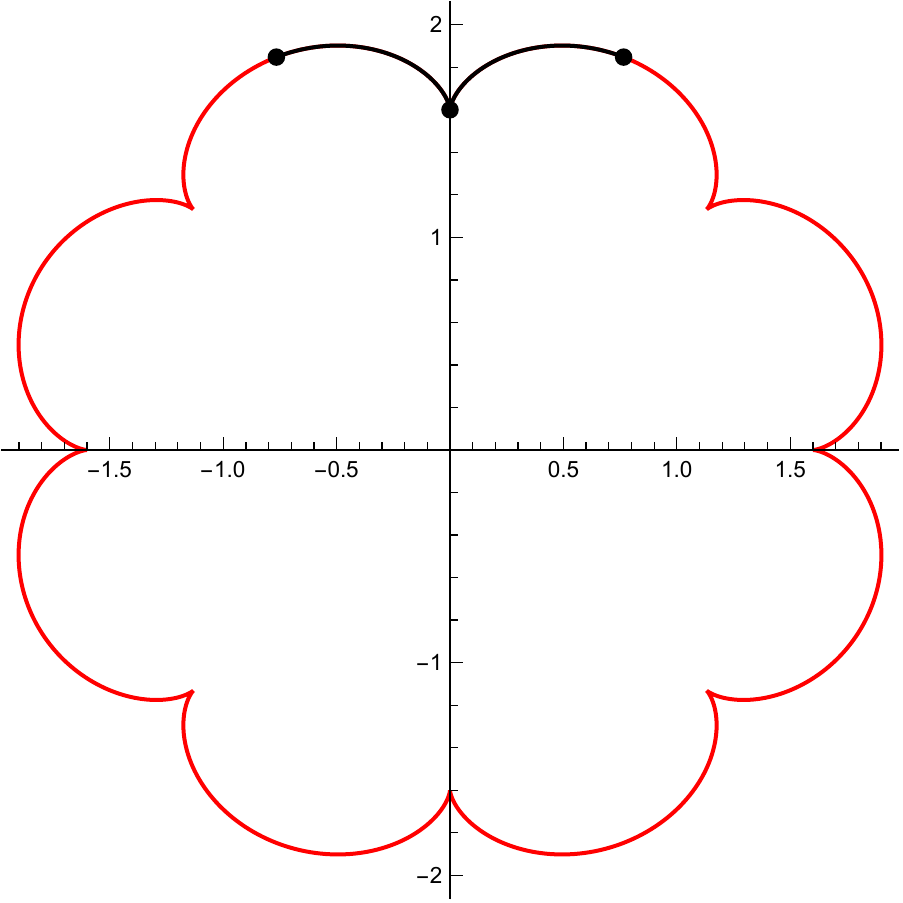}
\end{center}
\caption{Illustration of $\gamma_{1,9}^{-0.8}$ (left) as well as the rotation of the graph by $\phi=\pi/2-\pi/8$ (right). The black segment denotes the subcurve for $t\in [0,1/8]$ in both curves. The three black points are at $t=0, t=1/16, t=1/8$.} \label{rotation}
\end{figure}

As before we can switch to the representation of our curves in $\mathbb{R}^2$ and calculate the derivatives of the coordinate functions:
\begin{align*}
x_s'(t)&=-2a (1-s) \pi \cos\left( \frac{\pi}{b-a}-2a \pi t \right) - 2b (1+s) \pi \cos\left( \frac{\pi}{b-a}-2b \pi t \right)\\
y_s'(t)&=2a (1-s) \pi \sin\left( \frac{\pi}{b-a}-2a \pi t \right) + 2b (1+s) \pi \sin\left( \frac{\pi}{b-a}-2b \pi t \right).
\end{align*}
From this we can calculate the parametric derivative at the special value $s=\frac{a-b}{a+b}$ and obtain
$$\frac{y_s'(t)}{x_s'(t)} = - \tan \left( \pi \left( \frac{1}{b-a}-(a+b)t\right) \right). $$

Using the assumption $a=1$ we see that the parametric derivative is not defined for $t=\frac{1}{2(b-a)}$ when restricting to
the interval $t \in [0, 1/(b-a)]$:
The function $\tan(\phi)$ has a pole whenever $\phi=(1/2+n)\pi$ for $n \in \mathbb{Z}$. Now, it is easy to see that
$$a=1, \ t=\frac{1}{2(b-a)} \quad \Rightarrow \quad \frac{1}{b-a}-(a+b)t  = -\frac{1}{2}$$
Checking the derivative of the second coordinate function we see that $y'(t)=0$ and, hence, there is indeed a singular point.
Finally, to confirm that this point is indeed a cusp point we look again at the unit tangent vectors at this point when we approach $t$ from above and from below. 
We have that
\begin{align*}
\lim_{t\rightarrow )^+} &\frac{(\gamma_{1,b}^{s})'(t)}{\| (\gamma_{1,b}^{s})'(t)\|} = (0,1),\\
\lim_{t\rightarrow )^-} &\frac{(\gamma_{1,b}^{s})'(t)}{\| (\gamma_{1,b}^{s})'(t)\|} = (0,-1).\\
\end{align*}
Hence, we conclude that this singular point is indeed a cusp point and by symmetry we have a cusp point at every $t=\frac{h}{2(b-a)}$ for odd integers $h \in [0, \ldots, 2(b-a)]$. This finishes the proof of Theorem \ref{thm:general}.

What remains to show is that these are indeed the only cusp points and that a similar argument also works for the case $a>1$. We leave that to the interested reader. 

\section{Conclusion}
The main goal of this note was to provide a simple way to generate non-trivial examples of curves with arbitrary symmetry group $D_k$ and/or arbitrary number $k$ of cusp points as well as arbitrary winding number $k$. Our examples can be easily generated and plotted with any computer algebra system giving an abundance of different examples immediately.
We see the main application of our results in providing simple textbook and classroom examples of advanced mathematical concepts. From a mathematical point of view, an interesting future direction could be to investigate graphs of curves $\gamma_{a,b}$ for non-integer parameters $a$ and $b$.

\section*{Acknowledgements}
This work was done within the Mathematical Research department of TWT GmbH Science \& Innovation. TWT and the authors would like to thank Stefan Steinerberger for help with Lemma 2.

\end{document}